\documentclass{article}

\usepackage{arxiv}
\usepackage{amsmath}
\usepackage{amssymb}
 \usepackage{amscd}

\usepackage[utf8]{inputenc} % allow utf-8 input
\usepackage[T1]{fontenc}    % use 8-bit T1 fonts
\usepackage{url}            % simple URL typesetting
\usepackage{booktabs}       % professional-quality tables
\usepackage{amsfonts}       % blackboard math symbols
\usepackage{nicefrac}       % compact symbols for 1/2, etc.
\usepackage{microtype}      % microtypography
\usepackage{graphicx}
\usepackage{doi}

\newtheorem{theorem}{Theorem}[section]
\newtheorem{corollary}[theorem]{Corollary}

\newtheorem{definition}[theorem]{Definition}
\newtheorem{example}{Example}
\newtheorem{lemma}[theorem]{Lemma}

\newtheorem{remark}[theorem]{Remark}

\newenvironment{proof}[1][Proof]{\textbf{#1.} }{\ \rule{0.5em}{0.5em}}

\title{On the doubly stochastic realization of spectra}

\author{ Issam Kaddoura \\
	Department of Mathematics,\\ Lebanese International University,\\  Saida, Lebanon\\ 
	\texttt{issam.kaddoura@liu.edu.lb} \\
	\And
	Bassam Mourad \\
	Department of Mathematics,\\  Faculty of Science, Lebanese University,\\ Beirut, Lebanon\\
	\texttt{ bmourad@ul.edu.lb} \\
   }
\begin{document}

\title{  On a special class of block matrices and some applications}

\maketitle

\begin{abstract}
In this paper,  we study a particular class of block matrices placing an emphasis on their spectral properties.  Some related applications are then presented.
\end{abstract}
\textbf{AMS CLASSIFICATION}: 15A18, 15A51,15A24, 15A29, 12D10.

\section{Introduction}
Recall that every element $\pi$
in the permutation group $S_n$ corresponds to a permutation matrix
 $ P_{\pi}=(p_{ij})$ given by
  $$ p_{ij} =\left\{
\begin{array}
[c]{cc}
1 & \text{if} \ j=\pi(i)\\
 &            \\
0 & \ \text{ otherwise}.
\end{array}
\right . $$

 Let $\mathbb{C}$,  $\mathbb{Z}$  and $\mathbb{N}$ denote the field of complex  numbers,  the ring of integers and the set of natural numbers, respectively.  The $n\times n$ diagonal matrix whose $i$-th diagonal entry is $u_i$ will be denoted by diag$(u_1,...,u_n)$.
 A square matrix having in each row and column only one nonzero element is called a \emph{generalized permutation} (or \emph{monomial}) matrix.

 For any vector $u=[u_1,...,u_n]^T$ in $\mathbb{C}^n$, denote by $D_u=$diag$(u_1,...,u_n)$ and for any permutation $\pi\in S_n$, let  $\pi(u)=[u_{\pi(1)},...,u_{\pi(n)}]^T\in \mathbb{C}^n$. If $v=[v_1,...,v_n]^T\in \mathbb{C}^n$, then the Hadamard product of $u$ and $v$  which is denoted by $u\odot v$, is defined in the normal way as  $u\odot v=[u_1v_1,...,v_nv_n]^T$.

 It is easy to see that an $n\times n$ matrix $U$ is a generalized permutation matrix if and only if $U=D_uP_{\pi}$, in which $u\in \mathbb{C}^n$, $\pi\in S_n$ and $D_u=$diag$(u_1,...,u_n)$. One may also check that
\begin{equation}P_{\pi}D_u= D_{\pi(u)}P_{\pi}.\end{equation}

Next, recall that \emph{the order} of a permutation $\pi_{s},$ which is denoted by $O(\pi_{s})$, is defined as the smallest
positive integer $d$ such that
$\pi_{s}^{d}=\pi_0$  where $\pi_0$ is the identity in  $S_n$.
It is clear that if $d$ is the order of a permutation $\pi_s$ and  $P_{\pi_{s}}$ is the permutation matrix corresponding to it,
then $P_{\pi_{s}}^{d}=I_n$ where $I_n$ is the $n\times n$ identity matrix.

 The following lemma is needed for our purposes where its proof can be found in \cite{kad}.
\begin{lemma}
 Let $s\in\{0,1,2,3.....,n-1\}$,  $n\in \mathbb{N}$ and  $\pi_{s}\in S_n$ such that  $\pi_{s}(i)\equiv i+s\operatorname{mod}(n).$ Then,
\begin{enumerate}
\item
$
d=O(\pi_{s})=\frac{\operatorname{lcm}(n,s)}{s}=\frac{n}{\gcd(n,s)}.
$
\item  $n$  is prime if only if  $O(\pi_{s})=n$.
 \item  If  $\gcd(n,s)=1$,  then $O(\pi_{s})=n$.
\item   If
$O(\pi_{s})=\frac{\operatorname{lcm}(n,s)}{s}=d,$ then
$
\pi_{s}^{qd+r}=\pi_{s}^{r}%
$
for any integer $q$ and $0\leq r<d.$
\item  If  $O(\pi_{s})=d$, then $\pi_{s}
^{r_{1}}=\pi_{s}^{r_{2}}$ if and only if $r_{1}= r_{2}( \operatorname{mod} d).$

\end{enumerate}
\end{lemma}

 If $ P=(p_{ij})$ is the permutation matrix corresponding to the permutation  $\pi_{1}(i)\equiv i+1\operatorname{mod}(n)$  i.e. $P$ is  the $n\times n$  permutation matrix with 1 in the positions $(i,i+1)$ and $(n, 1)$  for all $i=1,\ldots n-1$, then  a circulant matrix $C\in M_n(\mathbb{C})$ is an element in the algebra generated by $P$ i.e. $C$ is of the form $$ C=c_0I_n+c_1P+\ldots +c_{n-1}P^{n-1} $$ for some complex numbers $c_0,...,c_{n-1}$. As the eigenvalues and eigenvectors of $P$ are well known, so are those of any circulant matrix $C$.

 In \cite{kad}, we generalized the class of circulant matrices to another class of  matrices which we called generalized weighted circulant matrices  as follows.
Let $G_w$ be the class of all generalized permutation matrices corresponding
to the following family of permutations $\pi_{s}(i)\equiv
i+s\operatorname{mod}(n)$ where $s\in w:=\{0,1,2,3.....,n-1\}$ and with $\pi_{0}(i)=\pi_{n}(i)\equiv i\operatorname{mod}(n)$ is the identity in $S_n$. Then, a generalized weighted circulant matrix ia an element in the algebra generated by an element $T$ in $G_w$. The main objective in \cite{kad} was to find  explicit expressions for the eigenvalues and eigenvectors of any $T\in G_w$, and to present some applications based on these formulas.

In this work, we shall generalize the class of  generalized weighted circulant matrices to what we call generalized weighted circulant block matrices. Roughly speaking, the idea is to start with any generalized permutation matrix $X$ in $G_w$  and replace the nonzero entries in $X$ by square block matrices of the same size to obtain a block generalized permutation matrix $M_X$ and then a generalized weighted circulant block matrix is an element in the algebra generated by $M_X$. Our main objective here is to study the spectral properties of this new class of matrices and to present some related applications as well.

We shall start by generalizing the notation introduced at the beginning to block matrices. If $A=(A_{ij})$ is a matrix partitioned into blocks, then to avoid any confusion we shall sometimes write  $A(i,j)$ to refer to the block $A_{ij}$ specifically when $i$ and $j$ are long expressions.

%Let \begin{eqnarray} A=\left[
%\begin{array}[c]{cccc}
%A_{11} & A_{12} &  \ldots & A_{1n}\\ \vdots & \vdots & \vdots & \vdots \\ A_{n1} & A_{n2} &  \ldots & A_{nn} \end{array} \right] \end{eqnarray} be  a block matrix where each block $A_{ij}$ is of size $k$. For such $A$, we denote by $D_A=$diag$(A_{11},...,A_{nn})$  to be the diagonal block matrix whose  $i$th diagonal block is $A_ii$. For any permutation $\pi\in S_n$, let  $D_{\pi(A)}=$diag$(A_{1\pi(1)},...,A_{n\pi(n)})$.
Recall that $A$ is said to be a \emph{block generalized permutation} matrix (BGPM) if $A$ has only one nonzero block in each block row and block column.
It is easy to check that an $n\times n$ block matrix $A$ having each block of size $k\times k$ is a block generalized permutation matrix with nonzero blocks $A_{1.\pi(1)},\ldots,A_{n.\pi(n)} $ if and only if $A=D_A(P_{\pi}\otimes I_k)$ where $D_A=\texttt{diag} (A_{1,\pi(1)},...,A_{n,\pi(n)})=\bigoplus \limits_{i=1}^{n}A_{i,\pi(i)}$ and $\otimes$ stands for the Kronecker product of matrices. Similar to (1), an inspection shows that for any  $\pi,  \partial \in S_n $, the following equality is valid
\begin{equation}
 \biggl(P_{\pi}\otimes I_k\biggr)\bigoplus \limits_{i=1}^{n}A_{i,\partial(i)}=\bigoplus\limits_{i=1}^{n}A_{\pi(i), \partial\circ\pi(i)}  \biggl(P_{\pi}\otimes I_k\biggr)
 \end{equation}

%We start with the following observation which will be of interest to us.
%% \begin{lemma} Let $A$ be as in (1). Then $A$ is GBPM if and only if
% $A=(P_{\pi}\otimes I_k)D_A=D_{\pi(A)}(P_{\pi}\otimes I_k)$.
%\end{lemma}
%\begin{proof}
%Note that $P_{\pi}\otimes I_k$ and $P^T_{\pi}\otimes I_k$ are $k\times k$ BGPM. Recall that conjugating a diagonal matrix  by a permutation matrix $P$  is only a permutation of the diagonal entries according to the permutation represented by the matrix $P$. Here the situation is the same but instead of entries we have blocks. Indeed, multiplying to the left, the ith block row of $(P_{\pi}\otimes I_k)$  of $D_A$ and   the right by the ith block colum
%\end{proof}

Next, we present the following useful lemma that exhibits a formula
for multiplying distinct block generalized permutation matrices.

\begin{lemma} Suppose that  $U_{\pi}=\left(
{\displaystyle\bigoplus\limits_{i=1}^{n}}
U(i,\pi(i))\right)  \biggl(P_{\pi}\otimes I_k\biggr)$ and $V_{\partial}=
\left({\displaystyle\bigoplus\limits_{i=1}^{n}}
V(i,\partial(i))\right)  \biggl(P_{\partial}\otimes I_k\biggr)$ are two  block generalized $kn\times kn$
permutation matrices corresponding to the permutations $\pi$ and $\partial$
respectively,% where $A(i,\pi(i))$ and $B(i,\partial(i))$ are mxm matrices for $i=1,2,3,...,n$,
 \ then
$$
U_{\pi}V_{\partial}    =\left(
%
%BeginExpansion
{\displaystyle\bigoplus\limits_{i=1}^{n}}
%EndExpansion
U_{i,\pi(i)}V_{\pi(i),\partial\circ\pi(i)}\right) \biggl( P_{\partial\circ\pi}\otimes I_k\biggr),$$  \mbox{ and  similarly }
$V_{\partial}U_{\pi}   =\left(
%
%BeginExpansion
{\displaystyle\bigoplus\limits_{i=1}^{n}}
%EndExpansion
V_{i,\partial(i)}U_{\partial(i),\pi\circ\partial(i)}\right)  \biggl(P_{\pi\circ\partial}\otimes I_k\biggr).
$
\end{lemma}

\begin{proof}  By making use of (2), we can write
\begin{align*}
U_{\pi}V_{\partial}  &  =\left(
{\displaystyle\bigoplus\limits_{i=1}^{n}}
U(i,\pi(i))\right)  \biggl(P_{\pi}\otimes I_k\biggr) \left({\displaystyle\bigoplus\limits_{i=1}^{n}}
V(i,\partial(i))\right)  \biggl(P_{\partial}\otimes I_k\biggr)\\
 &  =\left(
{\displaystyle\bigoplus\limits_{i=1}^{n}}
U(i,\pi(i))\right)  \left({\displaystyle\bigoplus\limits_{i=1}^{n}}
V(\pi(i),\partial\circ\pi (i))\right)  \biggl(P_{\pi}\otimes I_k\biggr)  \biggl(P_{\partial}\otimes I_k\biggr)\\
&  =\left(
{\displaystyle\bigoplus\limits_{i=1}^{n}}
%EndExpansion
U_{i,\pi(i)}V_{\pi(i),\partial\circ\pi(i)}\right) \biggl(P_{\partial\circ\pi}\otimes I_k\biggr).
\end{align*}
Of course, the second equality is now obvious.
\end{proof}

\begin{lemma}Suppose that  $U_{\pi}=\left(
{\displaystyle\bigoplus\limits_{i=1}^{n}}
U(i,\pi(i))\right)  \biggl(P_{\pi}\otimes I_k\biggr)$ is a BGPM. Then for any positive integer $r$, we have that
 $U_{\pi}^{r}=\left(
%TCIMACRO{\dbigoplus \limits_{i=1}^{n}}%
%BeginExpansion
{\displaystyle\bigoplus\limits_{i=1}^{n}}
%EndExpansion%
%TCIMACRO{\dprod \limits_{j=1}^{j=k}}%
%BeginExpansion
{\displaystyle\prod\limits_{j=1}^{r}}
%EndExpansion
U(\pi^{j-1}(i),\pi^{j}(i))\right)  \biggl(P_{\pi ^r}\otimes I_k\biggr).$
\end{lemma}

\begin{proof}  We proceed by induction on $r$. Obviously,  it is true for $r=1$.
Suppose now that the formula is true for $r-1$  i.e.
$$ U_{\pi}^{r-1}=\left(
%TCIMACRO{\dbigoplus \limits_{i=1}^{n}}%
%BeginExpansion
{\displaystyle\bigoplus\limits_{i=1}^{n}}
%EndExpansion%
%TCIMACRO{\dprod \limits_{j=1}^{r-1}}%
%BeginExpansion
{\displaystyle\prod\limits_{j=1}^{r-1}}
%EndExpansion
U(\pi^{j-1}(i),\pi^{j}(i))\right)   \biggl(P_{\pi ^{r-1}}\otimes I_k\biggr)=\left(
%
%BeginExpansion
{\displaystyle\bigoplus\limits_{i=1}^{n}}
%EndExpansion
B(i,\partial(i))\right)  \biggl(P_{\partial}\otimes I_k\biggr)$$
where $B(i,\partial(i))=%
%TCIMACRO{\dprod \limits_{j=1}^{j=r-1}}%
%BeginExpansion
{\displaystyle\prod\limits_{j=1}^{r-1}}
%EndExpansion
U(\pi^{j-1}(i),\pi^{j}(i))$ and $\partial=\pi^{r-1}.$
Now, making use of the previous lemma we obtain
\begin{align*}
U_{\pi}^{r}  &  = U_{\pi}.U_{\pi}^{r-1}=\left(
%
%BeginExpansion
{\displaystyle\bigoplus\limits_{i=1}^{n}}
%EndExpansion
U_{i,\pi(i)}B_{\pi(i),\partial\circ\pi(i)}\right) \biggl(P_{\partial\circ\pi}\otimes I_k\biggr)  \\
&  =\left(
%
%BeginExpansion
{\displaystyle\bigoplus\limits_{i=1}^{n}}
%EndExpansion
U_{i,\pi(i)}B_{\pi(i),\pi^{r-1}\circ\pi(i)}\right)   \biggl(P_{\pi ^{r-1}\circ\pi}\otimes I_k\biggr)   =\left(
%
%BeginExpansion
{\displaystyle\bigoplus\limits_{i=1}^{n}}
%EndExpansion
U_{i,\pi(i)}B_{\pi(i),\pi^{r}(i)}\right) \biggl(P_{\pi ^{r}}\otimes I_k\biggr)\\
&  =\left(
%
%BeginExpansion
{\displaystyle\bigoplus\limits_{i=1}^{n}}
%EndExpansion
U(i,\pi(i)).%
%TCIMACRO{\dprod \limits_{j=1}^{j=r-1}}%
{\displaystyle\prod\limits_{j=1}^{r-1}}
%EndExpansion
U(\pi^{j-1}(\pi(i)),\pi^{j}(\pi(i)))\right)   \biggl(P_{\pi ^{r}}\otimes I_k\biggr)\\
&    =\left(
{\displaystyle\bigoplus\limits_{i=1}^{n}}
%EndExpansion
U(i,\pi(i)).%
%TCIMACRO{\dprod \limits_{j=1}^{j=r-1}}%
%BeginExpansion
{\displaystyle\prod\limits_{j=1}^{r-1}}
%EndExpansion
U(\pi^{j}(i),\pi^{j+1}(i))\right)   \biggl(P_{\pi ^{r}}\otimes I_k\biggr) \\
&  =\left(
{\displaystyle\bigoplus\limits_{i=1}^{n}}
%EndExpansion
U(i,\pi(i)).%
%TCIMACRO{\dprod \limits_{j-1=1}^{j-1=r-1}}%
{\displaystyle\prod\limits_{j-1=1}^{r-1}}
%EndExpansion
U(\pi^{j-1}(i),\pi^{j}(i))\right)   \biggl(P_{\pi ^{r}}\otimes I_k\biggr) \\
&        =\left(
{\displaystyle\bigoplus\limits_{i=1}^{n}}
%EndExpansion
U(i,\pi(i)).%
%TCIMACRO{\dprod \limits_{j=2}^{j=r}}%
%BeginExpansion
{\displaystyle\prod\limits_{j=2}^{r}}
%EndExpansion
U(\pi^{j-1}(i),\pi^{j}(i))\right)   \biggl(P_{\pi ^{r}}\otimes I_k\biggr)   \\
&  =\left(
{\displaystyle\bigoplus\limits_{i=1}^{n}}
%TCIMACRO{\dprod \limits_{j=1}^{j=r}}%
%BeginExpansion
{\displaystyle\prod\limits_{j=1}^{r}}
%EndExpansion
U(\pi^{j-1}(i),\pi^{j}(i))\right)   \biggl(P_{\pi ^{r}}\otimes I_k\biggr),%
\end{align*}
and the proof is complete.
\end{proof}

\begin{corollary} If $P_{\pi_s}$ is a permutation matrix corresponding to
$\pi_{s}(i)\equiv i+s\operatorname{mod}(n)$  where $s\in \{0,1,2,3.....,n-1\}$, then
\[
\biggl(P_{\pi_{s}}\otimes I_k\biggr)^{O(\pi_{s})}=I.
\]
\end{corollary}
\begin{proof}
 $\biggl(P_{\pi_{s}}\otimes I_k\biggr)^{O(\pi_{s})}= \biggl(P_{\pi_{s}}\biggr) ^{O(\pi_{s})}\otimes \biggl( I_k\biggr)^{O(\pi_{s})}
=I_n\otimes I_k=I_{nk}.$
\end{proof}

Now, we are in a position to present formulaes for the eigenvalues and the characteristic polynomial of  any
block generalized permutation matrix corresponding to the family of the
permutations $\pi_{s}(i)\equiv i+s\operatorname{mod}(n),s\in \{0,1,2,3.....,n-1\}$ \ and $n\in \mathbb{N}.$
\begin{theorem} Let $U_{\pi_s}=\left(
{\displaystyle\bigoplus\limits_{i=1}^{n}}
U(i,\pi_s(i))\right)  \biggl(P_{\pi_s}\otimes I_k\biggr)$ be a $kn\times kn$ block generalized permutation
matrix where $s \in
\{1,2,3.....,n-1\},$ $\pi_{s}(i)\equiv i+s\operatorname{mod}(n),$  and
$U(i,\pi_{s}(i)) $ is a  $k\times k$ matrix for each $i=1,2,3,..,n.$  Moreover, let $d=O(\pi_{s})$ and $g=\gcd(n,s)$. Then,

$$ U_{\pi_{s}}^{d}  =%
%TCIMACRO{\dbigoplus \limits_{i=1}^{n=dg}}%
%BeginExpansion
{\displaystyle\bigoplus\limits_{i=1}^{dg}}
%EndExpansion
\biggl(%
%TCIMACRO{\dprod \limits_{j=1}^{j=d}}%
%BeginExpansion
{\displaystyle\prod\limits_{j=1}^{d}}
%EndExpansion
U(i+(j-1)s,i+js)\biggr).$$
Therefore,
the characteristic polynomial $p(x)$ of$\ U_{\pi_{s}}^{d}$ is given by

\[
\ \ \ p(x)=%
%TCIMACRO{\dprod \limits_{i=1}^{\gcd(n,s)}}%
%BeginExpansion
{\displaystyle\prod\limits_{i=1}^{g}}
%EndExpansion
\det\left(
%TCIMACRO{\dprod \limits_{j=1}^{j=d}}%
%BeginExpansion
{\displaystyle\prod\limits_{j=1}^{d}}
%EndExpansion
U\bigl(\pi_{s}^{j-1}(i),\pi_{s}^{j}(i)\bigr)-xI_{k}\right)  ^{d},
\]

and hence the eigenvalues $\mu(U_{\pi_{s}})$ of $U_{\pi_{s}}$ are determined by solving%

\[
\mu^{d}(U_{\pi_{s}})=\lambda_{i}(U_{\pi_{s}}^{d}), \ \ \ i=1,2,..g.
\]
\end{theorem}
\begin{proof}
From the preceding lemma,  we know that   $$U_{\pi}^{d}=\left(
%TCIMACRO{\dbigoplus \limits_{i=1}^{n}}%
%BeginExpansion
{\displaystyle\bigoplus\limits_{i=1}^{n}}
%EndExpansion%
%TCIMACRO{\dprod \limits_{j=1}^{j=k}}%
%BeginExpansion
{\displaystyle\prod\limits_{j=1}^{d}}
%EndExpansion
U(\pi_s^{j-1}(i),\pi_s^{j}(i))\right)  \biggl(P_{\pi ^d}\otimes I_k\biggr),$$ and  from
 Lemma 1.1,  we conclude that $n=gd$ and $P_{\pi_{s}}^{d}=I.$  This implies that
$$ U_{\pi_{s}}^{d}    =%
%TCIMACRO{\dbigoplus \limits_{i=1}^{n=dg}}%
%BeginExpansion
{\displaystyle\bigoplus\limits_{i=1}^{dg}}
%EndExpansion
\left(
%TCIMACRO{\dprod \limits_{j=1}^{j=d}}%
%BeginExpansion
{\displaystyle\prod\limits_{j=1}^{d}}
%EndExpansion
U(i+(j-1)s,i+js)\right). $$
It is well known    that the diagonal blocks of  $U_{\pi_s}^{d}$  which are
$
%TCIMACRO{\dprod \limits_{j=1}^{d}}%
%BeginExpansion
{\displaystyle\prod\limits_{j=1}^{d}}
%EndExpansion
U(i+(j-1)s,i+js),$ \\
$%TCIMACRO{\dprod \limits_{j=1}^{d}}%
%BeginExpansion
{\displaystyle\prod\limits_{j=1}^{d}}
%EndExpansion
U(i+g+(j-1)s,i+g+js),\ldots,%
%TCIMACRO{\dprod \limits_{j=1}^{d}}%
%BeginExpansion
{\displaystyle\prod\limits_{j=1}^{d}}
%EndExpansion
U(i+(d-1)g+(j-1)s,i+(d-1)g+js)$\\
have the same eigenvalues.  This in turn means that the following $d$ determinants are equal\\
$ \det\left(
%TCIMACRO{\dprod \limits_{j=1}^{d}}%
%BeginExpansion
{\displaystyle\prod\limits_{j=1}^{d}}
%EndExpansion
U(i+(j-1)s,i+js)-xI_{k}\right)
 =\det\left(
%TCIMACRO{\dprod \limits_{j=1}^{d}}%
%BeginExpansion
{\displaystyle\prod\limits_{j=1}^{d}}
%EndExpansion
U(i+g+(j-1)s,i+g+js)-xI_{k}\right) $\\
  $=\ldots=\det\left(
%TCIMACRO{\dprod \limits_{j=1}^{d}}%
%BeginExpansion
{\displaystyle\prod\limits_{j=1}^{d}}
%EndExpansion
U(i+(d-1)g+(j-1)s,i+(d-1)g+js)-xI_{k}\right).
$
Consequently, the characteristic polynomial of$\ U_{\pi_{s}}^{d}$ is given by
\[
\ \ \ p(x)=%
%TCIMACRO{\dprod \limits_{i=1}^{g}}%
%BeginExpansion
{\displaystyle\prod\limits_{i=1}^{g}}
%EndExpansion
\det\left(
%TCIMACRO{\dprod \limits_{j=1}^{d}}%
%BeginExpansion
{\displaystyle\prod\limits_{j=1}^{d}}
%EndExpansion
A(\pi_{s}^{j-1}(i),\pi_{s}^{j}(i))-xI_{k}\right)  ^{d}.
\]
This completes the proof.
\end{proof}

The particular case $s=1$ gives the following result which is Theorem 1 in  Mirsky\cite{mir} (see also \cite{pet}).
\begin{corollary} Let \begin{eqnarray} U_{\pi_1}=\left(
\begin{array}{ccccc}
0  & U_1 & 0 & \ldots & 0 \\
0 & 0 & U_2 \ \ & \ddots  &  \vdots \\
\vdots &  \vdots     & \ddots& \ddots & 0\\
 0   & 0       & \ldots  &  0  & U_{n-1}  \\
 U_n & 0        & \ldots        &  0  &  0 \\
\end{array}
\right)
\end{eqnarray} where the  $U_i$ are $k\times k $ matrices. If the eigenvalues of the $k\times k $ matrix  $U_1U_2...U_n$ are $w_1, \ w_2,\ldots, \ w_k $. Then,
the eigenvalues of $U_{\pi_1}$ are the $kn$ numbers consisting of all
$n$-th roots of each of the numbers $w_1, \ w_2, \ ,..., \ w_k $.
\end{corollary}

 To illustrate our results, we start with the following simple example.
\begin{example}  Consider $n=4, \ k=2, \ s=2,$ and $\pi_{2}(i)\equiv
i+2\operatorname{mod}4.$  Let  $U_{\pi_{2}}$ be given by
$U_{1,\pi_{2}(1)}=U(1,3)=\left(
\begin{array}
[c]{cc}%
a & 0\\
0 & b
\end{array}
\right)  , \ U_{2,\pi_{2}(2)}=U(2,4)=\left(
\begin{array}
[c]{cc}%
c & 0\\
0 & d
\end{array}
\right)  ,$ \
$U_{3,\pi_{2}(3)}=U(3,1)=\left(
\begin{array}
[c]{cc}%
x & 0\\
0 & y
\end{array}
\right)  , \ U_{4,\pi_{2}(4)}=U(4,2)=\left(
\begin{array}
[c]{cc}%
z & 0\\
0 & t
\end{array}
\right),   $ so that
\[
U_{\pi_{2}}=\left(
\begin{array}
[c]{cccc}%
\left(
\begin{array}
[c]{cc}%
{\small 0} & {\small 0}\\
{\small 0} & {\small 0}%
\end{array}
\right)  & \left(
\begin{array}
[c]{cc}%
{\small 0} & {\small 0}\\
{\small 0} & {\small 0}%
\end{array}
\right)  & \left(
\begin{array}
[c]{cc}%
{\small a} & {\small 0}\\
{\small 0} & {\small b}%
\end{array}
\right)  & \left(
\begin{array}
[c]{cc}%
{\small 0} & {\small 0}\\
{\small 0} & {\small 0}%
\end{array}
\right) \\
\left(
\begin{array}
[c]{cc}%
{\small 0} & {\small 0}\\
{\small 0} & {\small 0}%
\end{array}
\right)  & \left(
\begin{array}
[c]{cc}%
{\small 0} & {\small 0}\\
{\small 0} & {\small 0}%
\end{array}
\right)  & \left(
\begin{array}
[c]{cc}%
{\small 0} & {\small 0}\\
{\small 0} & {\small 0}%
\end{array}
\right)  & \left(
\begin{array}
[c]{cc}%
{\small c} & {\small 0}\\
{\small 0} & {\small d}%
\end{array}
\right) \\
\left(
\begin{array}
[c]{cc}%
{\small x} & {\small 0}\\
{\small 0} & {\small }y%
\end{array}
\right)  & \left(
\begin{array}
[c]{cc}%
{\small 0} & {\small 0}\\
{\small 0} & {\small 0}%
\end{array}
\right)  & \left(
\begin{array}
[c]{cc}%
{\small 0} & {\small 0}\\
{\small 0} & {\small 0}%
\end{array}
\right)  & \left(
\begin{array}
[c]{cc}%
{\small 0} & {\small 0}\\
{\small 0} & {\small 0}%
\end{array}
\right) \\
\left(
\begin{array}
[c]{cc}%
{\small 0} & {\small 0}\\
{\small 0} & {\small 0}%
\end{array}
\right)  & \left(
\begin{array}
[c]{cc}%
{\small z} & {\small 0}\\
{\small 0} & {\small t}%
\end{array}
\right)  & \left(
\begin{array}
[c]{cc}%
{\small 0} & {\small 0}\\
{\small 0} & {\small 0}%
\end{array}
\right)  & \left(
\begin{array}
[c]{cc}%
{\small 0} & {\small 0}\\
{\small 0} & {\small 0}%
\end{array}
\right)
\end{array}
\right).
\]
Clearly, the order of $\pi_{2}$ is $d=O(\pi_{2})=\dfrac{\operatorname{lcm}(4,2)}{2}=2,$ and
$g=gcd(4,2)=2,$ so that in order to
 determine the spectrum of $U_{\pi_{2}},$ we first  need to determine the spectrum of
the block generalized permutation matrix $U_{\pi_{2}}^{2}.$ Obviously,
\begin{align*}
U_{\pi_{2}}^{2}  &  =%
%TCIMACRO{\dbigoplus \limits_{i=1}^{2}}%
%BeginExpansion
{\displaystyle\bigoplus\limits_{i=1}^{2}}
%EndExpansion%
%TCIMACRO{\dprod \limits_{j=1}^{2}}%
%BeginExpansion
{\displaystyle\prod\limits_{j=1}^{2}}
%EndExpansion
U(i+2(j-1),i+2j)%
%TCIMACRO{\dbigoplus \limits_{i=3}^{4}}%
%BeginExpansion
{\displaystyle\bigoplus\limits_{i=3}^{4}}
%EndExpansion%
%TCIMACRO{\dprod \limits_{j=1}^{2}}%
%BeginExpansion
{\displaystyle\prod\limits_{j=1}^{2}}
%EndExpansion
U(i+2(j-1)s,i+2j)\\
&  =
%TCIMACRO{\dbigoplus \limits_{i=1}^{2}}%
%BeginExpansion
{\displaystyle\bigoplus\limits_{i=1}^{2}}
%EndExpansion
U(i,i+2)U(i+2,i)%
%TCIMACRO{\dbigoplus \limits_{i=3}^{4}}%
%BeginExpansion
{\displaystyle\bigoplus\limits_{i=3}^{4}}
%EndExpansion
U(i,i+2)U(i+2,i)\\
&  =U(1,3)U(3,1)%
%TCIMACRO{\dbigoplus }%
%BeginExpansion
{\displaystyle\bigoplus}
%EndExpansion
U(2,4)U(4,2)%
%TCIMACRO{\dbigoplus }%
%BeginExpansion
{\displaystyle\bigoplus}
%EndExpansion
U(3,1)U(1,3)%
%TCIMACRO{\dbigoplus }%
%BeginExpansion
{\displaystyle\bigoplus}
%EndExpansion
U(4,2)U(2,4))\\
&  =\allowbreak\left(
\begin{array}
[c]{cc}%
{\small ax} & {\small 0}\\
{\small 0} & {\small by}%
\end{array}
\right)
%TCIMACRO{\dbigoplus }%
%BeginExpansion
{\displaystyle\bigoplus}
%EndExpansion
\allowbreak\left(
\begin{array}
[c]{cc}%
{\small cz} & {\small 0}\\
{\small 0} & {\small dt}%
\end{array}
\right)
%TCIMACRO{\dbigoplus }%
%BeginExpansion
{\displaystyle\bigoplus}
%EndExpansion
\allowbreak\left(
\begin{array}
[c]{cc}%
{\small ax} & {\small 0}\\
{\small 0} & {\small by}%
\end{array}
\right)
%TCIMACRO{\dbigoplus }%
%BeginExpansion
{\displaystyle\bigoplus}
%EndExpansion
\left(
\begin{array}
[c]{cc}%
{\small cz} & {\small 0}\\
{\small 0} & {\small dt}%
\end{array}
\right)
\end{align*}
Therefore, the eigenvalues of $U_{\pi_{2}}$ are  clearly  given by
$ \sqrt{ax}, \ -\sqrt{ax} , \ \sqrt{by}, \ -\sqrt{by}, \ \sqrt{cz}, \ -\sqrt{cz}, \ \sqrt{dt}, \ -\sqrt{dt}.$
\end{example}

Nest,  we take a similar numeric example but with minor changes namely  the nonzero diagonal blocks are taken to be non-diagonal.
\begin{example}  Consider $n=4, \ k=2, \ s=2,$ and $\pi_{2}(i)\equiv
i+2\operatorname{mod}4.$  Let  $U_{\pi_{2}}$ be given by
$U_{1,\pi_{2}(1)}=U(1,3)=\left(
\begin{array}
[c]{cc}%
1 & 2\\
1 & -1
\end{array}
\right)  , \ U_{2,\pi_{2}(2)}=U(2,4)=\left(
\begin{array}
[c]{cc}%
3 & 0\\
2 & -2
\end{array}
\right)  ,$ \
$U_{3,\pi_{2}(3)}=U(3,1)=\left(
\begin{array}
[c]{cc}%
2 & 5\\
0 & 0
\end{array}
\right)  , \ U_{4,\pi_{2}(4)}=U(4,2)=\left(
\begin{array}
[c]{cc}%
5 & 1\\
-1 & -1
\end{array}
\right),   $ so that
\[
U_{\pi_{2}}=\left(
\begin{array}
[c]{cccc}%
\left(
\begin{array}
[c]{cc}%
{\small 0} & {\small 0}\\
{\small 0} & {\small 0}%
\end{array}
\right)  & \left(
\begin{array}
[c]{cc}%
{\small 0} & {\small 0}\\
{\small 0} & {\small 0}%
\end{array}
\right)  & \left(
\begin{array}
[c]{cc}%
{\small 1} & {\small 2}\\
{\small 1} & {\small -1}%
\end{array}
\right)  & \left(
\begin{array}
[c]{cc}%
{\small 0} & {\small 0}\\
{\small 0} & {\small 0}%
\end{array}
\right) \\
\left(
\begin{array}
[c]{cc}%
{\small 0} & {\small 0}\\
{\small 0} & {\small 0}%
\end{array}
\right)  & \left(
\begin{array}
[c]{cc}%
{\small 0} & {\small 0}\\
{\small 0} & {\small 0}%
\end{array}
\right)  & \left(
\begin{array}
[c]{cc}%
{\small 0} & {\small 0}\\
{\small 0} & {\small 0}%
\end{array}
\right)  & \left(
\begin{array}
[c]{cc}%
{\small 3} & {\small 0}\\
{\small 2} & {\small -2}%
\end{array}
\right) \\
\left(
\begin{array}
[c]{cc}%
{\small 2} & {\small 5}\\
{\small 0} & {\small 0}%
\end{array}
\right)  & \left(
\begin{array}
[c]{cc}%
{\small 0} & {\small 0}\\
{\small 0} & {\small 0}%
\end{array}
\right)  & \left(
\begin{array}
[c]{cc}%
{\small 0} & {\small 0}\\
{\small 0} & {\small 0}%
\end{array}
\right)  & \left(
\begin{array}
[c]{cc}%
{\small 0} & {\small 0}\\
{\small 0} & {\small 0}%
\end{array}
\right) \\
\left(
\begin{array}
[c]{cc}%
{\small 0} & {\small 0}\\
{\small 0} & {\small 0}%
\end{array}
\right)  & \left(
\begin{array}
[c]{cc}%
{\small 5} & {\small 1}\\
{\small -1} & {\small -1}%
\end{array}
\right)  & \left(
\begin{array}
[c]{cc}%
{\small 0} & {\small 0}\\
{\small 0} & {\small 0}%
\end{array}
\right)  & \left(
\begin{array}
[c]{cc}%
{\small 0} & {\small 0}\\
{\small 0} & {\small 0}%
\end{array}
\right)
\end{array}
\right).
\]
Clearly, the order of $\pi_{2}$ is $d=O(\pi_{2})=\dfrac{\operatorname{lcm}(4,2)}{2}=2,$ and
$g=gcd(4,2)=2,$ so that in order to
 determine the spectrum of $U_{\pi_{2}},$ we first  need to determine the spectrum of
the block generalized permutation matrix $U_{\pi_{2}}^{2}.$ Obviously,
\begin{align*}
U_{\pi_{2}}^{2}  &  =%
%TCIMACRO{\dbigoplus \limits_{i=1}^{2}}%
%BeginExpansion
{\displaystyle\bigoplus\limits_{i=1}^{2}}
%EndExpansion%
%TCIMACRO{\dprod \limits_{j=1}^{2}}%
%BeginExpansion
{\displaystyle\prod\limits_{j=1}^{2}}
%EndExpansion
U(i+2(j-1),i+2j)%
%TCIMACRO{\dbigoplus \limits_{i=3}^{4}}%
%BeginExpansion
{\displaystyle\bigoplus\limits_{i=3}^{4}}
%EndExpansion%
%TCIMACRO{\dprod \limits_{j=1}^{2}}%
%BeginExpansion
{\displaystyle\prod\limits_{j=1}^{2}}
%EndExpansion
U(i+2(j-1)s,i+2j)\\
&  =
%TCIMACRO{\dbigoplus \limits_{i=1}^{2}}%
%BeginExpansion
{\displaystyle\bigoplus\limits_{i=1}^{2}}
%EndExpansion
U(i,i+2)U(i+2,i)%
%TCIMACRO{\dbigoplus \limits_{i=3}^{4}}%
%BeginExpansion
{\displaystyle\bigoplus\limits_{i=3}^{4}}
%EndExpansion
U(i,i+2)U(i+2,i)\\
&  =U(1,3)U(3,1)%
%TCIMACRO{\dbigoplus }%
%BeginExpansion
{\displaystyle\bigoplus}
%EndExpansion
U(2,4)U(4,2)%
%TCIMACRO{\dbigoplus }%
%BeginExpansion
{\displaystyle\bigoplus}
%EndExpansion
U(3,1)U(1,3)%
%TCIMACRO{\dbigoplus }%
%BeginExpansion
{\displaystyle\bigoplus}
%EndExpansion
U(4,2)U(2,4))\\
&  =\allowbreak\left(
\begin{array}
[c]{cc}%
{\small 2} & {\small 5}\\
{\small 2} & {\small 5}%
\end{array}
\right)
%TCIMACRO{\dbigoplus }%
%BeginExpansion
{\displaystyle\bigoplus}
%EndExpansion
\allowbreak\left(
\begin{array}
[c]{cc}%
{\small 15} & {\small 3}\\
{\small 12} & {\small 4}%
\end{array}
\right)
%TCIMACRO{\dbigoplus }%
%BeginExpansion
{\displaystyle\bigoplus}
%EndExpansion
\allowbreak\left(
\begin{array}
[c]{cc}%
{\small 7} & {\small -1}\\
{\small 0} & {\small 0}%
\end{array}
\right)
%TCIMACRO{\dbigoplus }%
%BeginExpansion
{\displaystyle\bigoplus}
%EndExpansion
\left(
\begin{array}
[c]{cc}%
{\small 17} & {\small -2}\\
{\small -5} & {\small 2}%
\end{array}
\right)
\end{align*}

In other words,

\[
U_{\pi_{2}}^{2}=\left(
\begin{array}
[c]{cccccccc}%
{\small 2} & {\small 5} & {\small 0} & {\small 0} & {\small 0} & {\small 0} &
{\small 0} & {\small 0}\\
{\small 2} & {\small 5} & {\small 0} & {\small 0} & {\small 0} & {\small 0} &
{\small 0} & {\small 0}\\
{\small 0} & {\small 0} & {\small 15} & {\small 3} & {\small 0} & {\small 0} &
{\small 0} & {\small 0}\\
{\small 0} & {\small 0} & {\small 12} & {\small 4} & {\small 0} & {\small 0} &
{\small 0} & {\small 0}\\
{\small 0} & {\small 0} & {\small 0} & {\small 0} & {\small 7} & {\small -1} &
{\small 0} & {\small 0}\\
{\small 0} & {\small 0} & {\small 0} & {\small 0} & {\small 0} & {\small 0} &
{\small 0} & {\small 0}\\
{\small 0} & {\small 0} & {\small 0} & {\small 0} & {\small 0} & {\small 0} &
{\small 17} & {\small -2}\\
{\small 0} & {\small 0} & {\small 0} & {\small 0} & {\small 0} & {\small 0} &
{\small -5} & {\small 2}%
\end{array}
\right),
\]
and a simple check shows that the characteristic polynomial  of $U_{\pi_{2}}^{2}$ is given by
$p(x)= x^{2}\left(  x-7\right)  ^{2}\left(  -19x+x^{2}+24\right)  ^{2}.$ Therefore,
the eigenvalues of $U_{\pi_{2}}^{2}$ can be exhibited as $ \frac{19}{2}+\frac{1}%
{2}\sqrt{265}, \ \frac{19}{2}+\frac{1}{2}\sqrt{265}, \ \frac{19}{2}-\frac{1}{2}%
\sqrt{265}, \ \frac{19}{2}-\frac{1}{2}\sqrt{265}, \ 0, \ 0, \ 7, \ 7.$ However,
 this in turn means that the eigenvalues of $U_{\pi_{2}}$ are  given by
$ \frac{1}{2}\sqrt{2}\sqrt{19-\sqrt{265}%
}, \ -\frac{1}{2}\sqrt{2}\sqrt{19-\sqrt{265}}, \ \frac{1}{2}\sqrt{2}%
\sqrt{\sqrt{265}+19}, \ -\frac{1}{2}\sqrt{2}\sqrt{\sqrt{265}+19}, \ 0, \ 0, \ \sqrt{7},\\ -\sqrt{7}.$
\end{example}

\begin{definition}  Let $U_{\pi_s}=\left(
{\displaystyle\bigoplus\limits_{i=1}^{n}}
U(i,\pi_s(i))\right)  \biggl(P_{\pi_s}\otimes I_k\biggr)$ be a $kn\times kn$ block generalized permutation
matrix where $s \in
\{1,2,3.....,n-1\},$ $\pi_{s}(i)\equiv i+s\operatorname{mod}(n),$  and
$U(i,\pi_{s}(i)) $ is a $k\times k$ matrix for each $i=1,2,3,..,n.$  Then,
a block generalized permutation circulant matrix $C$ is an element in the algebra generated by
$U_{\pi_{s}}$  i.e.
\[
C=\sum_{r=0}^{m}c_{r}(U_{\pi_{s}})^{r},
\]
where $m\in \mathbb{N}$  and  the $c_{r}$ are complex constants.
\end{definition}

 Making use  of  the preceding theorem,  we conclude the following.
\begin{theorem} Let $U_{\pi_s}=\left(
{\displaystyle\bigoplus\limits_{i=1}^{n}}
U(i,\pi_s(i))\right)  \biggl(P_{\pi_s}\otimes I_k\biggr)$ be a $kn\times kn$ block generalized permutation
matrix where $s \in
\{1,2,3.....,n-1\},$ $\pi_{s}(i)\equiv i+s\operatorname{mod}(n),$  and
$U(i,\pi_{s}(i)) $ is a $k\times k$ matrix for each $i=1,2,3,..,n.$ Moreover, let  $d=O(\pi_{s})$ and
$g=\gcd(n,s). $
If  $C=\sum_{r=0}^{m}c_{r}(U_{\pi_{s}})^{r}$ \ is a generalized permutation
circulant matrix generated by $U_{\pi_{s}}$, then the eigenvalues of $C$ are given by
\[
\lambda_{t,r}(C)=\sum_{r=0}^{m}c_{r}\left(  \sqrt[d]{\lambda_{t}}\right)
^{r}, \ t=1,2,\ldots,g.
\]
where $\lambda_{t}$ is an eigenvalue of $
{\displaystyle\prod\limits_{j=1}^{d}}
U(\pi_{s}^{j-1}(t),\pi_{s}^{j}(t))$ for $t=1,2,..., g.$
\end{theorem}

\subsection{Application 1}

 Recall that the inverse eigenvalue problem for special types of matrices is concerned with constructing a
matrix that maintains the required structure from its set of eigenvalues.
 Constructing matrices for the inverse eigenvalue problems arise in a remarkable variety of applications, including system and control theory, geophysics, molecular spectroscopy, quantum mechanics, particle physics, structure analysis, economics and operation research to name a few (see, e.g. \cite{chu}).

 An $n\times n$ matrix $A$  is said to be {\em  nonnegative } if all entries in $A$ are nonnegative.  An $n\times n$ nonnegative matrix $A$ whose
each row and column sum  equals to 1 is said to be {\em doubly stochastic}.
If $J_{n}$  denotes the
$n\times n$ matrix whose all entries are $\frac{1}{n}$, then clearly  an $n\times n$ nonnegative matrix $A$  is doubly stochastic if and only if $AJ_n=J_nA=J_n.$
 The nonnegative inverse eigenvalue problem (NIEP) (see, e.g., \cite{johnson}) is concerned with finding the necessary and sufficient conditions on a set of $n$ complex numbers to be the spectrum of an $n\times n$ nonnegative matrix $A$.
%As mentioned in Egleston 2004, many sub-problems have emerged from the NIEP because of its complexity. One of these is the symmetric nonnegative inverse eigenvalue problem (SNIEP) which deals with finding which sets of $n$ real numbers serve as the spectrum of an $n\times n$ symmetric nonnegative matrix $A$.
Similarly, the doubly stochastic inverse eigenvalue problem (DIEP)  deals with finding necessary and sufficient conditions on a  set of $n$ complex numbers in order to be the spectrum of an $n\times n$ doubly stochastic matrix $A$.

 Next, we shall show that Theorem 1.5 and Theorem 1.7  can serve as building blocks
for obtaining many new results concerning the preceding two problems. More explicitly,
one of the obvious applications of Theorem 1.5 is concerned with
generating new sufficient conditions from known ones for both the NIEP and  DIEP.
More generally,  any known results concerning these problems can lead to
new results as we shall illustrate below. 
%In addition,  the main importance of the methods introduced here lies in the fact that they might help as a checking
%tool in case of a conjecture concerning complete solutions is given for these two problems.

 \subsubsection{nonnegative matrices}
  Here, we shall show how Theorem 1.5 can be used to find new sufficient conditions for the resolutions of the NIEP. In addition, we shall see how any known sufficient conditions for the NIEP in low dimension can lead to  new sufficient conditions for the NIEP in higher dimension.
 We shall start by proving that for any positive integer $n$, the $n$-th roots of the eigenvalues of nonnegative matrices can occur as spectra of nonnegative matrices.
\begin{theorem}
Let $A$ be a $k\times k$ nonnegative matrix with eigenvalues
$\lambda_1,$ $\lambda_2,$... $\lambda_k$. For each positive integer $n$ and for each $i=1,...,k$,  let  $\mu_1^{(i)},$ $\mu_2^{(i)},$...,$\mu_n^{(i)}$
be the $n$th roots of $\lambda_i$, then there exists a $kn\times kn$  nonnegative matrix with eigenvalues
$$\{ \mu_1^{(1)}, \mu_2^{(1)},..., \mu_n^{(1)},\ldots, \mu_1^{(k)},\mu_2^{(k)},...,\mu_n^{(k)}\}.$$ Moreover, we recursively conclude that for any positive integer $s$, the $s$th roots of the $n$th roots of $\{\lambda_1,$ $\lambda_2,$... $\lambda_k\}$   can occur as spectra  of nonnegative matrices.
\end{theorem}
\begin{proof}
By Corollary1,6,
 it suffices to take \begin{eqnarray} U_{\pi_1}(n)=\left(
\begin{array}{ccccc}
0  & U_1 & 0 & \ldots & 0 \\
0 & 0 & U_2 \ \ & \ddots  &  \vdots \\
\vdots &  \vdots     & \ddots& \ddots & 0\\
 0   & 0       & \ldots  &  0  & U_{n-1}  \\
 U_n & 0        & \ldots        &  0  &  0 \\
\end{array}
\right)
\end{eqnarray} where   $U_n=A$  and $U_i=I_k$  for each $i=1,...,n-1$ and then we obtain the nonnegative matrix
 $T_n(A)=\left(
\begin{array}{ccccc}
0  & I_k & 0 & \ldots & 0 \\
0 & 0 & I_k\ \ & \ddots  &  \vdots \\
\vdots &  \vdots     & \ddots& \ddots & 0\\
 0   & 0       & \ldots  &  0  & I_k  \\
 A & 0        & \ldots        &  0  &  0 \\
\end{array}
\right)$ with the required spectrum.   In order to prove the second part,  we now consider $U_{\pi_1}(s)$ and then we repeat the same process this time with $U_s=T_n(A)$ and $U_i=I_n$   for each $i=1,...,s-1$ to obtain the required result.
\end{proof}

Next, we shall see
how any known sufficient conditions  for the NIEP can serve as building blocks for obtaining many new ones. Although, one can  systematically collect all the known  sufficient conditions for the NIEP and use it to obtain new ones, however this will not be done here and instead  we restrict ourselves only to two examples to illustrate the idea.
 Indeed, we shall illustrate our results by considering  the following
theorem whose proof can be found in \cite{minc}.
\begin{theorem}
 Let $\lambda_1 > 0\geq \lambda_2\geq ...\geq \lambda_k .$  Then there exists a $k\times k$ nonnegative  matrix whose eigenvalues  are $\{\lambda_1,..., \lambda_n\}$.
\end{theorem}

\begin{corollary} Let $\lambda_1 > 0\geq \lambda_2\geq ...\geq \lambda_k .$ Then for any positive integer $n$, there exists a $kn\times kn$ matrix  $A$ whose eigenvalues are
the $n$th roots of  $\{ \lambda_1,...,\lambda_k \}.$ More generally,  for any positive integer $s$, the $s$th roots of the $n$th roots of $\{\lambda_1,$ $\lambda_2,$... $\lambda_k\}$   can occur as spectra  of nonnegative matrices.

\end{corollary}
In \cite{rammal} we have proved the following theorem.
\begin{theorem} 
Let  $\left(\lambda_2,\ldots, \lambda_k\right)$ be any list of complex numbers which is closed under complex conjugation and let  the real and imaginary parts of $\lambda_i$  be respectively denoted by $x_i=\Re(\lambda_i)$ and  $y_i=\Im(\lambda_i)$. Then, the following holds.\\
 1) Suppose that $\left(\lambda_2,\ldots, \lambda_k\right)$ are such that  $\lambda_2\geq \ldots\geq \lambda_p   $   are real for $2\leq p\leq k$ and $\lambda_{p+1},\ldots,\lambda_k$ are non-real. Moreover, let $$w= \min\{\pm \lambda_2,\ldots,\pm \lambda_p,\pm x_{p+1}, \pm y_{p+1} ,\ldots, \pm x_\frac{k+p}{2}, \pm y_\frac{k+p}{2} \}.$$  Then, there exists a $k\times k$ nonnegative matrix with eigenvalues $\left((k-1)|w|, \lambda_2,\ldots, \lambda_k\right)$. \\
 2) Suppose now that $\left(\lambda_2,\ldots, \lambda_k\right)$ with $k=2h-1$, is such that none of its components  is real and let $w= \min\{\pm x_{2}, \pm y_{2} ,\ldots, \pm x_{h}, \pm y_{h} \},$ Then, there exists a $k\times k$ nonnegative matrix with eigenvalues $\left(k|w|, \lambda_2,\ldots, \lambda_k\right)$.
\end{theorem}

\begin{corollary} Let  $\left(\lambda_2,\ldots, \lambda_k\right)$ be any list of complex numbers which is closed under complex conjugation and let  the real and imaginary parts of $\lambda_i$  be respectively denoted by $x_i=\Re(\lambda_i)$ and  $y_i=\Im(\lambda_i)$.
 Then, the following holds.\\
 1) Suppose that $\left(\lambda_2,\ldots, \lambda_k\right)$ are such that  $\lambda_2\geq \ldots\geq \lambda_p   $ are real  for $2\leq p\leq k$ and $\lambda_{p+1},\ldots,\lambda_k$ are non-real.  Moreover, let $$w= \min\{\pm \lambda_2,\ldots,\pm \lambda_p,\pm x_{p+1}, \pm y_{p+1} ,\ldots, \pm x_\frac{k+p}{2}, \pm y_\frac{+kp}{2} \}.$$  Then, for any positive integer $n$, there exists a $kn\times kn$ matrix  $A$ whose eigenvalues are
the $n$th roots of   $\left((k-1)|w|, \lambda_2,\ldots, \lambda_k\right)$. More generally,  for any positive integer $s$, the $s$th roots of the $n$th roots of  $\left((k-1)|w|, \lambda_2,\ldots, \lambda_k\right)$   can occur as spectra  of nonnegative matrices. \\
 2) Suppose now that $\left(\lambda_2,\ldots, \lambda_k\right)$ with $k=2h-1$, is such that none of its components  is real and  let $w= \min\{\pm x_{2}, \pm y_{2} ,\ldots, \pm x_{h}, \pm y_{h} \},$ Then, for any positive integer $n$, there exists a $kn\times kn$ matrix  $A$ whose eigenvalues are
the $n$th roots of   $\left(k|w|, \lambda_2,\ldots, \lambda_k\right)$. More generally,  for any positive integer $s$, the $s$th roots of the $n$th roots of  $\left(k|w|, \lambda_2,\ldots, \lambda_k\right)$   can occur as spectra  of nonnegative matrices. 
\end{corollary}

\subsubsection{doubly stochastic matrices}

We start here by noting that in the proof of Theorem 1.8 if $A$ happens to be doubly stochastic then the matrix $T_n(A)$ is also doubly stochastic. Thus we have the following.
\begin{theorem}
Let $A$ be a $k\times k$ doubly stochastic  matrix with eigenvalues
$\lambda_1,$ $\lambda_2,$... $\lambda_k$. For each $i=1,...,k$,  let  $\mu_1^{(i)},$ $\mu_2^{(i)},$...,$\mu_n^{(i)}$
be the $n$th roots of $\lambda_i$, then there exists an $kn\times kn$  doubly stochastic  matrix with eigenvalues
$$\{ \mu_1^{(1)}, \mu_2^{(1)},..., \mu_n^{(1)},\ldots, \mu_1^{(k)},\mu_2^{(k)},...,\mu_n^{(k)}\}.$$ Moreover, we recursively conclude that for any positive integer $s$, the $s$th roots of the $n$th roots of $A$ can occur as spectra  of doubly stochastic matrices.
\end{theorem}

In \cite{mourad}, we have proved the following theorem.
\begin{theorem} For $n=2k$,  then the $n$-list $(1,0,...,0,-1)$ is the spectrum of an $n\times n$ doubly stochastic matrix, and for $n=2k+1$,
the $n$-list $(1,0,...,0,-1)$ can not be the spectrum of an $n\times n$ doubly stochastic matrix,
\end{theorem}

Noting that $\{ 1, -1\}$ are the square roots of unity, we prove the following generalization of the preceding theorem.

\begin{theorem} Let $k>1$ be any positive integer and Let $w=e^{\frac{2\pi i}{k}}$ be a primitive $k$th root of unity. Then, the
list $(1,0,...,0,w, w^2,...,w^{k-1})$ is the spectrum of $kn\times kn$ doubly stochastic matrix.
\end{theorem}
\begin{proof}
In view of  Corollary1,6,
 it suffices to take \begin{eqnarray} U_{\pi_1}(n)=\left(
\begin{array}{ccccc}
0  & U_1 & 0 & \ldots & 0 \\
0 & 0 & U_2 \ \ & \ddots  &  \vdots \\
\vdots &  \vdots     & \ddots& \ddots & 0\\
 0   & 0       & \ldots  &  0  & U_{n-1}  \\
 U_n & 0        & \ldots        &  0  &  0 \\
\end{array}
\right)
\end{eqnarray} where   $U_n=J_k$  and $U_i=I_k$ (or any $k\times k$ doubly stochastic matrix as in this case $U_1.U_2... U_{n-1}J_k=J_k$)  for each $i=1,...,n-1$ and then we obtain a doubly stochastic matrix $M$
whose eigenvalues are the $n$th roots of the eigenvalues of $J_k$ which are $\{1,0,...,0\}$ and thus $M$  satisfies the requirements of the theorem.
\end{proof}

In a similar fashion to the case of the NIEP, any known sufficient conditions for the DIEP, will essentially  lead to new results. We illustrate our discussion in 3 situations.
The first one  solves the DIEP for the case
$n=2.$
\begin{lemma}
 There exists a $2\times 2$ symmetric  doubly stochastic matrix
 with spectrum $(1,\lambda)$  if and only if
  $-1\leq \lambda \leq 1.$
\end{lemma}
\begin{proof} It suffices to note that for $-1\leq \lambda \leq 1,$
the matrix $ X=\left(
\begin{array}{cc}
\frac{1+\lambda}{2} & \frac{1-\lambda}{2}   \\
\frac{1-\lambda}{2} & \frac{1+\lambda}{2} \\
\end{array}
\right)$ is doubly stochastic and has spectrum $(1,\lambda).$
\end{proof}

As a conclusion, we have the following.
\begin{corollary}
Let $-1\leq \lambda \leq 1.$ Then, for positive integers $n$ and $s$, the $s$th roots of the $n$tth roots of $\{1, \lambda\}$  occur as the spectra of doubly stochastic matrices.
\end{corollary}
 The second situation deals with the case $n=3$ (see, e.g. \cite{mourad1, perfect}) and can be stated as follows.
\begin{theorem}
 There exists a symmetric $3\times 3$ doubly stochastic matrix
 with spectrum $(1,\lambda,\mu )$ if and only if
  $-1\leq \lambda \leq 1,$ $-1\leq \mu\leq 1,$
$\lambda+3\mu+2\geq 0$ and $3\lambda+\mu+2\geq 0.$
\end{theorem}
\begin{proof} Suppose that $ \lambda \geq \mu,$ then
the matrix $ X=\frac{1}{6}\left(
\begin{array}{ccc}
2+4\lambda & 2-2\lambda &  2-2\lambda \\
2-2\lambda & 2+\lambda+3\mu & 2+\lambda-3\mu \\
2-2\lambda & 2+\lambda-3\mu & 2+\lambda+3\mu
\end{array}
\right)$ is doubly
stochastic and has spectrum $(1,\lambda,\mu).$ Now if $ \mu \geq
\lambda,$ then the proof can be completed by exchanging the roles
of $\lambda$ and $\mu.$
\end{proof}

As a result, we have the following.
\begin{corollary}
Suppose that $-1\leq \lambda \leq 1,$ $-1\leq \mu\leq 1,$
$\lambda+3\mu+2\geq 0$ and $3\lambda+\mu+2\geq 0.$ Then, for positive integers $n$ and $s$, the $s$th roots of the $n$tth roots of $\{1, \lambda, \mu\}$  occur as the spectra of doubly stochastic matrices.
\end{corollary}
 Now if we let $\Pi_3$ denote the region of the complex plane specified as the convex hull
   of the  cubic roots of unity. Now our third illustration  stems from the following
   theorem which is due to ~\cite{perfect} and the proof which we
   include here appears in ~\cite{minc}.
\begin{theorem}~\cite{perfect}
Let $z$ be a complex number and $\bar{z}$ be its complex
conjugate. Then $(1,z,\bar{z})$ is the spectrum of a $3\times 3$
doubly stochastic matrix if and only if $z\in\Pi_3.$
\end{theorem}
\begin{proof} It suffices to check that
the circulant matrix
$$ X=\frac{1}{3}\left(
\begin{array}{ccc}
1+2r\cos\theta & 1-2r\cos(\frac{\pi}{3}+\theta) &  1-2r\cos(\frac{\pi}{3}-\theta) \\
1-2r\cos(\frac{\pi}{3}-\theta) & 1+2r\cos\theta & 1-2r\cos(\frac{\pi}{3}+\theta) \\
1-2r\cos(\frac{\pi}{3}+\theta) & 1-2r\cos(\frac{\pi}{3}-\theta) &
1+2r\cos\theta
\end{array}
\right)$$  has spectrum
$(1,z=re^{i\theta},\bar{z}=re^{-i\theta}),$ and $X$ is doubly
stochastic if and only if $ z=re^{i\theta} \in \Pi_3.$
\end{proof}

Consequently, we have the following.
\begin{corollary}
Let  $z\in\Pi_3.$ Then, for positive integers $n$ and $s$, the $s$th roots of the $n$tth roots of $\{1,z, \bar{z} \}$  occur as the spectra of doubly stochastic matrices.
\end{corollary}
\begin{remark}
It is worthy to note that in our applications here, we have only used  the permutation $\pi_{1}\in S_n$ which is defined by  $\pi_{1}(i)\equiv i+1\operatorname{mod}(n).$
Of course, when using different permutations we can clearly obtain new results by making use of the same techniques.
\end{remark}

\subsection{Application 2}

Recall that Fermat's last theorem states that
the equation $\ x^{p}+y^{p}=z^{p}$ has no nontrivial solutions over the
integers for all $p>2$.  More generally,
Beal's conjecture  states that the equation $x^{p}+y^{q}=z^{r}$ has no nontrivial
solutions where $p, \ q$ and $r$ are integers greater
than 2 and $x,y$ and $z$ are coprime integers (see for example \cite{ma,rib}).

Not much is known concerning  integral matrix solutions to the equations $X^{p}+Y^{p}=Z^{p}$  and $X^{p}+Y^{q}=Z^{r}$ where
 $X$, $Y$ and  $Z$ are $n\times n$ integer  matrices with a particular structure. More generally, in this subsection
 we shall present some results concerning finding integer matrix solutions to the two equations $aX^{p}+bY^{p}=cZ^{p}$  and $aX^{p}+bY^{q}=cZ^{r}$ for any integers $a, b$ and $c$. By making use of certain properties of generalized permutation matrices, we are able to
prove the existence of an infinite number of solutions for these two problems such that  these solution matrices have all of their entries  nonzero
natural numbers.

 We shall begin with the following theorem.
\begin{theorem} Let $a$, $b$ and $c$ be any integers. Then, for every positive integer $k$, there is an infinite number of
$n\times n$ square matrices $A,$ $B$ and $C$ with integer entries such that $aA^{k}
+bB^{k}=cC^{k}.$ In particular, we can select all entries of $A,$ $B$ and $C$ to be nonzero natural numbers.
\end{theorem}

\begin{proof}  As usual, let $\ \pi_{s}(i)\equiv i+s\operatorname{mod}n, \ s\in$
$\{1,2,3..,n-1\}.$  Now choose $s_{1},s_{2},s_{3}\in$ $\{1,2,3.....,n-1\}$
such that \ $ks_{1}\equiv ks_{2}\equiv ks_{3}\operatorname{mod}n$ (obviously
our choice is valid at least for $s_{1}=s_{2}=s_{3}),$
then $\ \pi_{s_{1}}^{k}=\ \pi_{s_{2}}^{k}=\ \pi_{s_{3}}^{k}$ and
$P_{\pi_{s_{1}}^{k}}=P_{\pi_{s_{2}}^{k}}=P_{\pi_{s_{3}}^{k}}.$ Next, consider the
following three generalized permutation  block
$mn\times mn$  matrices:\\
$A_{\pi_{s_1}}=\left(
{\displaystyle\bigoplus\limits_{i=1}^{n}}
A(i,\pi_{s_1}(i))\right)  \biggl(P_{\pi_{s_1}}\otimes I_m\biggr)$, \  $B_{\pi_{s_2}}=\left(
{\displaystyle\bigoplus\limits_{i=1}^{n}}
B(i,\pi_{s_2}(i))\right)  \biggl(P_{\pi_{s_2}}\otimes I_m\biggr)$
 and  $C_{\pi_{s_3}}=\left(
{\displaystyle\bigoplus\limits_{i=1}^{n}}
C(i,\pi_{s_3}(i))\right)  \biggl(P_{\pi_{s_3}}\otimes I_m\biggr)$
  where
$A(i,\pi_{s_1}(i)),B(i,\pi_{s_2}(i))$  and $C(i,\pi_{s_{3}}(i))$ are $m\times m$  nonzero matrices whose nonzero entries are parameters  for $i=1,...,n$.   From Lemma 1.3, we get
$$ A_{\pi_{s_{1}}}^{k}=
\biggl(%
%TCIMACRO{\dbigoplus \limits_{i=1}^{i=n}}%
%BeginExpansion
{\displaystyle\bigoplus\limits_{i=1}^{n}}
%EndExpansion
%TCIMACRO{\dprod \limits_{j=1}^{j=k}}%
%BeginExpansion
{\displaystyle\prod\limits_{j=1}^{k}}
%EndExpansion
A(i+(j-1)s_{1},i+js_{1})\biggr)\biggl(P_{\pi_{s_1}^k}\otimes I_m\biggr),$$

$$ B_{\pi_{s_{2}}}^{k}=
\biggl(%
%TCIMACRO{\dbigoplus \limits_{i=1}^{i=n}}%
%BeginExpansion
{\displaystyle\bigoplus\limits_{i=1}^{n}}
%EndExpansion
%TCIMACRO{\dprod \limits_{j=1}^{j=k}}%
%BeginExpansion
{\displaystyle\prod\limits_{j=1}^{k}}
%EndExpansion
B(i+(j-1)s_{2},i+js_{2})\biggr)\biggl(P_{\pi_{s_2}^k}\otimes I_m\biggr),$$
$$ C_{\pi_{s_{3}}}^{k}=
\biggl(%
%TCIMACRO{\dbigoplus \limits_{i=1}^{i=n}}%
%BeginExpansion
{\displaystyle\bigoplus\limits_{i=1}^{n}}
%EndExpansion
%TCIMACRO{\dprod \limits_{j=1}^{j=k}}%
%BeginExpansion
{\displaystyle\prod\limits_{j=1}^{k}}
%EndExpansion
C(i+(j-1)s_{3},i+js_{3})\biggr)\biggl(P_{\pi_{s_3}^k}\otimes I_m\biggr).$$  Since $\ \pi_{s_{1}}^{k}=\ \pi_{s_{2}}^{k}=\ \pi_{s_{3}}^{k}$, then
solving the following system  of equations,
\begin{align*}
 a\biggl(%
%TCIMACRO{\dbigoplus \limits_{i=1}^{i=n}}%
%BeginExpansion
{\displaystyle\bigoplus\limits_{i=1}^{n}}
%EndExpansion
%TCIMACRO{\dprod \limits_{j=1}^{j=k}}%
%BeginExpansion
{\displaystyle\prod\limits_{j=1}^{k}}
%EndExpansion
A(i+(j-1)s_{1},i+js_{1})\biggr) + b \biggl(%
%TCIMACRO{\dbigoplus \limits_{i=1}^{i=n}}%
%BeginExpansion
{\displaystyle\bigoplus\limits_{i=1}^{n}}
%EndExpansion
%TCIMACRO{\dprod \limits_{j=1}^{j=k}}%
%BeginExpansion
{\displaystyle\prod\limits_{j=1}^{k}}
%EndExpansion
B(i+(j-1)s_{2},i+js_{2})\biggr)  &= \ \ \ \ \ \ \ \ \ \ \ \ \   \\
    c\biggl(%
%TCIMACRO{\dbigoplus \limits_{i=1}^{i=n}}%
%BeginExpansion
{\displaystyle\bigoplus\limits_{i=1}^{n}}
%EndExpansion
%TCIMACRO{\dprod \limits_{j=1}^{j=k}}%
%BeginExpansion
{\displaystyle\prod\limits_{j=1}^{k}}
%EndExpansion
C(i+(j-1)s_{3},i+js_{3})\biggr)
\end{align*}
 over $\mathbb{Z}$,
yields the desired result $aA_{\pi_{s_1}}^{k}+bB_{\pi_{s_2}}^{k}=cC_{\pi_{s_3}}^{k}.$

An inspection shows that we can always  find some relations between the parameters of these solution matrices to obtain an appropriate subfamily of solutions
$T, R$ and $W$ which are pairwise commuting matrices and all of whose  nonzero entries are positive integers.
Now, let $D=p(T,R,W)$ be any polynomial in  $T, R$ and $W$ whose all entries are nonzero natural numbers. Obviously, $D$ commutes with the matrices
$T, R$ and $W$.  Clearly, the matrices $A= DT,$  $B = DR$ and $C= DW$ have all their entries as nonzero natural numbers
and  $aA^{k}+bB^{k}-cC^{k}=$ $a\left(  DT\right)
^{k}+b\left(  DR\right)  ^{k}-c\left(  DW\right)  ^{k}=0$.
\end{proof}

As a conclusion, we have  the following.
\begin{corollary}
If $n$ a is prime, then for any integers $a$, $b$ and $c$ there is an infinite number of $n\times n$ matrices $A, \ B$ and
$C$ with integer entries such that $aA^{n}+bB^{n}=cC^{n}.$
\end{corollary}
Another conclusion is given in the following.
\begin{corollary}  Let $a$, $b$ and $c$ be any integers. Then the equation $aA^{n}+bB^{p}=cC^{q}$ \ where $n, \ p$ and $q$ are
arbitrary positive integers,  has an infinite family of matrix solutions over
the square matrices with integer entries. In particular, we can select all
entries in the solution matrices to be natural numbers.
\end{corollary}
\begin{proof} Let $k = lcm(n,p,q)$, then there exists positive integers $r, \ s$ and $t$ such
that $k=rn,\ k=sp$, and $k=tq$.
From the previous theorem, we know that there exists an infinite number of matrices such that
$aA^{k}+bB^{k}=cC^{k}.$ But this in turn means that $aA^{rn}+bB^{sp}=cC^{tq}.$ Therefore,  $
a\left (A^{r}\right)  ^{n}+ b\left( B^{s}\right)  ^{p}=  c\left (C^{t}\right)  ^{q},$
and the proof is complete.
\end{proof}
\begin{example}
In order to illustrate our results, we shall solve for positive integer matrices the following equation
$$4X^{5}=Y^{5}+3Z^{5}.$$

We start with following generalized permutation matrices whose all nonzero entries are considered as parameters

$$ T=\left(
\begin{array}
[c]{ccccc}
0 & m & 0 & 0 & 0\\
0 & 0 & n & 0 & 0\\
0 & 0 & 0 & p & 0\\
0 & 0 & 0 & 0 & q\\
r & 0 & 0 & 0 & 0
\end{array}
\right)  , \ R=\left(
\begin{array}
[c]{ccccc}%
0 & 0 & u & 0 & 0\\
0 & 0 & 0 & v & 0\\
0 & 0 & 0 & 0 & w\\
s & 0 & 0 & 0 & 0\\
0 & t & 0 & 0 & 0
\end{array}
\right)  , \ W=\left(
\begin{array}
[c]{ccccc}
0 & 0 & 0 & x & 0\\
0 & 0 & 0 & 0 & y\\
z & 0 & 0 & 0 & 0\\
0 & k & 0 & 0 & 0\\
0 & 0 & h & 0 & 0
\end{array}
\right)  $$
 We then find some relations among their entries so  that  $T, R$ and $W$ are pairwise commuting matrices. In particular, we consider the following setting

$$ T=\left(
\begin{array}
[c]{ccccc}
0 & tuw & 0 & 0 & 0\\
0 & 0 & suv & 0 & 0\\
0 & 0 & 0 & tvw & 0\\
0 & 0 & 0 & 0 & suw\\
stv & 0 & 0 & 0 & 0
\end{array}
\right)  ,  R=\left(
\begin{array}
[c]{ccccc}%
0 & 0 & u & 0 & 0\\
0 & 0 & 0 & v & 0\\
0 & 0 & 0 & 0 & w\\
s & 0 & 0 & 0 & 0\\
0 & t & 0 & 0 & 0
\end{array}
\right),  \mbox{  and }  $$  $$ W=\left(
\begin{array}
[c]{ccccc}%
0 & 0 & 0 & tuvw & 0\\
0 & 0 & 0 & 0 & suvw\\
stvw & 0 & 0 & 0 & 0\\
0 & stuw & 0 & 0 & 0\\
0 & 0 & stuv & 0 & 0
\end{array}
\right).  $$
By  solving the following system
$ 4T ^{5}-R^{5} -3W ^{5}=0 $
we obtain $stuvw-1=0$.  Taking  $s=\frac{1}{tuvw}$ and defining

$A=tuvw T =\left(
\begin{array}
[c]{ccccc}%
0 & t^{2}u^{2}vw^{2} & 0 & 0 & 0\\
0 & 0 & uv & 0 & 0\\
0 & 0 & 0 & t^{2}uv^{2}w^{2} & 0\\
0 & 0 & 0 & 0 & uw\\
tv & 0 & 0 & 0 & 0
\end{array}
\right)  $

$B=tuvw R =\left(
\begin{array}
[c]{ccccc}%
0 & 0 & tu^{2}vw & 0 & 0\\
0 & 0 & 0 & tuv^{2}w & 0\\
0 & 0 & 0 & 0 & tuvw^{2}\\
1 & 0 & 0 & 0 & 0\\
0 & t^{2}uvw & 0 & 0 & 0
\end{array}
\right)  $

$C=tuvw W =\left(
\begin{array}
[c]{ccccc}%
0 & 0 & 0 & t^{2}u^{2}v^{2}w^{2} & 0\\
0 & 0 & 0 & 0 & uvw\\
tvw & 0 & 0 & 0 & 0\\
0 & tuw & 0 & 0 & 0\\
0 & 0 & tuv & 0 & 0
\end{array}
\right) $ where $t, u, v$ and $w$ are now restricted to be in $\mathbb{N}$,
 then we clearly have   $4A^{5}-B^{5}-3C^{5}=0$.
 Next, let $r, a, b, c$ and $d$ be any positive integers,  and define the matrix
  \begin{align*} D &=rI_5  +aA +bB +cC +dC^3 \\
 &=\left(
\begin{array}
[c]{ccccc}%
r & at^{2}u^{2}vw^{2} & btu^{2}vw & ct^{2}u^{2}v^{2}w^{2} & dt^{3}u^{4}%
v^{3}w^{4}\\
dt^{2}u^{2}v^{3}w^{2} & r & auv & btuv^{2}w & cuvw\\
ctvw & dt^{4}u^{3}v^{3}w^{4} & r & at^{2}uv^{2}w^{2} & btuvw^{2}\\
b & ctuw & dt^{2}u^{3}v^{2}w^{2} & r & auw\\
atv & bt^{2}uvw & ctuv & dt^{4}u^{3}v^{4}w^{3} & r
\end{array}
\right)
\end{align*}
which obviously  commutes with $A, \ B$  and $C$. Up to this point, let\\
$X=DA=\left(
\begin{array}
[c]{ccccc}%
dt^{4}u^{4}v^{4}w^{4} & rt^{2}u^{2}vw^{2} & at^{2}u^{3}v^{2}w^{2} &
bt^{3}u^{3}v^{3}w^{3} & ct^{2}u^{3}v^{2}w^{3}\\
ctuv^{2}w & dt^{4}u^{4}v^{4}w^{4} & ruv & at^{2}u^{2}v^{3}w^{2} & btu^{2}%
v^{2}w^{2}\\
bt^{2}uv^{2}w^{2} & ct^{3}u^{2}v^{2}w^{3} & dt^{4}u^{4}v^{4}w^{4} &
rt^{2}uv^{2}w^{2} & at^{2}u^{2}v^{2}w^{3}\\
atuvw & bt^{2}u^{2}vw^{2} & ctu^{2}vw & dt^{4}u^{4}v^{4}w^{4} & ruw\\
rtv & at^{3}u^{2}v^{2}w^{2} & bt^{2}u^{2}v^{2}w & ct^{3}u^{2}v^{3}w^{2} &
dt^{4}u^{4}v^{4}w^{4}%
\end{array}
\right),  $\\
$Y=DB=\left(
\begin{array}
[c]{ccccc}%
ct^{2}u^{2}v^{2}w^{2} & dt^{5}u^{5}v^{4}w^{5} & rtu^{2}vw & at^{3}u^{3}%
v^{3}w^{3} & bt^{2}u^{3}v^{2}w^{3}\\
btuv^{2}w & ct^{2}u^{2}v^{2}w^{2} & dt^{3}u^{4}v^{4}w^{3} & rtuv^{2}w &
atu^{2}v^{2}w^{2}\\
at^{2}uv^{2}w^{2} & bt^{3}u^{2}v^{2}w^{3} & ct^{2}u^{2}v^{2}w^{2} &
dt^{5}u^{4}v^{5}w^{5} & rtuvw^{2}\\
r & at^{2}u^{2}vw^{2} & btu^{2}vw & ct^{2}u^{2}v^{2}w^{2} & dt^{3}u^{4}%
v^{3}w^{4}\\
dt^{4}u^{3}v^{4}w^{3} & rt^{2}uvw & at^{2}u^{2}v^{2}w & bt^{3}u^{2}v^{3}w^{2}
& ct^{2}u^{2}v^{2}w^{2}%
\end{array}
\right)  $ and \\
$Z=DC=\left(
\begin{array}
[c]{ccccc}%
bt^{2}u^{2}v^{2}w^{2} & ct^{3}u^{3}v^{2}w^{3} & dt^{4}u^{5}v^{4}w^{4} &
rt^{2}u^{2}v^{2}w^{2} & at^{2}u^{3}v^{2}w^{3}\\
atuv^{2}w & bt^{2}u^{2}v^{2}w^{2} & ctu^{2}v^{2}w & dt^{4}u^{4}v^{5}w^{4} &
ruvw\\
rtvw & at^{3}u^{2}v^{2}w^{3} & bt^{2}u^{2}v^{2}w^{2} & ct^{3}u^{2}v^{3}w^{3} &
dt^{4}u^{4}v^{4}w^{5}\\
dt^{3}u^{3}v^{3}w^{3} & rtuw & atu^{2}vw & bt^{2}u^{2}v^{2}w^{2} &
ctu^{2}vw^{2}\\
ct^{2}uv^{2}w & dt^{5}u^{4}v^{4}w^{4} & rtuv & at^{3}u^{2}v^{3}w^{2} &
bt^{2}u^{2}v^{2}w^{2}%
\end{array}
\right) ,$
then it is easy to check that
$4X^{5}-Y^{5}-3Z^{5}=0$.

\end{example}

\section*{Acknowledgments}

 The first author is supported from LIU.

\end{document}